\documentclass[12pt]{amsart}
\usepackage{amsfonts,amsmath,amscd,amssymb}
\usepackage{mathrsfs}
\usepackage{latexsym,graphicx,verbatim,nameref}
\usepackage[numbers,sort&compress]{natbib}
\usepackage[all,cmtip]{xy}
\usepackage[hidelinks]{hyperref}
\usepackage{fancyhdr}
\def\acts{\curvearrowright}
\def\var{\varepsilon}
\def\Complex{\mathbb{C}}

\newtheorem{theorem}{Theorem}[section]
\newtheorem{corollary}[theorem]{Corollary}
\newtheorem{lemma}[theorem]{Lemma}
\newtheorem{proposition}[theorem]{Proposition}
\newtheorem*{theorem*}{Theorem}
\newtheorem*{proposition*}{Proposition}
\theoremstyle{definition}
\newtheorem{definition}[theorem]{Definition}
\newtheorem{remark}[theorem]{Remark}

\begin{document}

\title{Compact quantum stabilizer subgroups}
\author{Huichi Huang}
\address{College  of Mathematics and Statistics, Chongqing University, Chongqing, 401332, PR. China}
\email{huanghuichi@cqu.edu.cn}
\keywords{Compact quantum group, stabilizer subgroup}
\subjclass[2010]{Primary: 46L65; Secondary:16W22}
\date{\today}
\begin{abstract}
We generalize the concept of  stabilizer subgroups to compact quantum groups.
\end{abstract}
\maketitle

\section{Introduction}
We introduce the concept of compact quantum stabilizer subgroups, which are generalizations of compact stabilizer subgroups.

This short note was finished during my PhD study. I put it on my personal webpage for several years and finally decide to put it on arxiv with no intent of submission in the near future.

\section{Formulation of definition}

Consider a compact quantum group $\mathcal{G}$ acting on a compact Hausdorff space $X$. Let $\mathcal{H}$  be a subgroup of $\mathcal{G}$ and $I_{\mathcal{H}}$ be the corresponding Woronowicz $C^*$-ideal of $\mathcal{H}$. Use $\pi_{\mathcal{H}}:A\rightarrow A/I_{\mathcal{H}}$ to denote morphism from $\mathcal{H}$ to $\mathcal{G}$, and $\Delta_{\mathcal{H}}$ to denote the coproduct of $\mathcal{H}$.

Use $\alpha_{\mathcal{H}}:B\rightarrow B\otimes (A/I_{\mathcal{H}})$ to denote the induced action of $\mathcal{H}$ on $X$, i.e.,
$$\alpha_{\mathcal{H}}(f)=(id\otimes\pi_{\mathcal{H}})\alpha(f)$$ for $f\in B$.

Let $\mathcal{G}/\mathcal{H}$ to be the {\it quotient space}, which is the dual space of the C*-algebra
$$C(\mathcal{G}/\mathcal{H})=\{a\in A|(\pi_{\mathcal{H}}\otimes id)\Delta(a)=1\otimes a\}.$$

Then $\mathcal{G}\acts \mathcal{G}/\mathcal{H}$ by $\Delta|_{C(\mathcal{G}/\mathcal{H})}$.

\begin{definition}
 Let $Y\subseteq X$. We say that $\mathcal{H}$ \textbf{fixes} $Y$ (under the action $\alpha$) if $$(ev_{y}\otimes id)\alpha_{\mathcal{H}}(f)=f(y)1$$ for all $f\in B$ and $y\in Y$, i.e., $(ev_{y}\otimes id)\alpha(f)-f(y)1$ is in $\ker{\pi_{\mathcal{H}}}$.  If a subgroup $\mathcal{H}_Y$ of $\mathcal{G}$ fixes $Y$ and every subgroup $\mathcal{H}$ of $\mathcal{G}$ fixing $Y$ is a subgroup of $\mathcal{H}_Y$, i.e., there is an epimorphism $\pi_{\mathcal{H},Y}:A/I_Y \rightarrow A/I_{\mathcal{H}}$ such that
$$(\pi_{\mathcal{H},Y}\otimes \pi_{\mathcal{H},Y})\Delta_Y=\Delta_{\mathcal{H}}\pi_{\mathcal{H},Y},$$ then $\mathcal{H}_Y$ is called the {\it stabilizer} of $Y$.
\end{definition}

\begin{remark}
If the counit of $\mathcal{G}$ is bounded, then $\Complex$ is a subgroup fixing any  $Y\subseteq X$.
\end{remark}

Suppose $Y$ is an $\alpha$-invariant subspace of $X$ containing $x$.
\begin{proposition}
$\mathcal{H}$ fixes $x$ under $\alpha$ if and only if $\mathcal{H}$ fixes $x$ under $\alpha_Y$.
\end{proposition}
\begin{proof}
Let $\widetilde{ev_{x}}$ to be evaluation functional of $C(Y)$ at $x$. Note that $C(Y)\cong C(X)/J_Y$.
Hence $(\widetilde{ev_{x}}\otimes id)\alpha_Y(f+J_Y)=(ev_{x}\otimes id)\alpha(f)$ for all $f\in B$. It follows that
$(\widetilde{ev_{x}}\otimes id)\alpha_Y(f+J_Y)=\widetilde{ev_{x}}(f+J_Y)1$ if and only if $(ev_{x}\otimes id)\alpha(f)=f(x)1$ for all $f\in B$, which completes the proof.
\end{proof}

\begin{proposition}\label{mx is a quotient of quotient space}
If $\mathcal{H}$ fixes $Y$, then $(ev_{x}\otimes id)\alpha(B)\subseteq C(\mathcal{G}/\mathcal{H})$ for all $x\in Y$.
\end{proposition}
\begin{proof}
Take any $f\in B$. We have
$$\Delta((ev_{x}\otimes id)\alpha(f))=(ev_{x}\otimes id\otimes id)(\alpha\otimes id)\alpha(f).$$
Applying $\pi_{\mathcal{H}}\otimes id$ on both sides of the above identity, we have
$$(\pi_{\mathcal{H}}\otimes id)\Delta((ev_{x}\otimes id)\alpha(f))=(((ev_{x_0}\otimes \pi_{\mathcal{H}})\alpha)\otimes id)\alpha(f).$$
Since $\mathcal{H}$ fixes $x$, we have  $(ev_{x}\otimes \pi_{\mathcal{H}})\alpha(\cdot)=ev_{x}(\cdot)1$.
This implies
$$(\pi_{\mathcal{H}}\otimes id)\Delta((ev_{x}\otimes id)\alpha(f))=((ev_{x}\otimes \pi_{\mathcal{H}})\alpha\otimes id)\alpha(f)=1\otimes ((ev_{x}\otimes id)\alpha(f)).$$
Therefore $(ev_{x}\otimes id)\alpha(f)\in C(\mathcal{G}/\mathcal{H})$ for all $f\in B$.
\end{proof}

\begin{corollary}\label{fixfor}
If $\mathcal{H}$ fixes $x$, then $(ev_{x}\otimes id)\alpha: B\rightarrow C(\mathcal{G}/\mathcal{H})$ is equivariant.
\end{corollary}
\begin{proof}
It follows from $(\alpha\otimes id)\alpha=(id\otimes \Delta)\alpha$ that
$$\Delta|_{C(\mathcal{G}/\mathcal{H})}(ev_{x}\otimes id)\alpha=[((ev_{x}\otimes id)\alpha)\otimes id]\alpha.$$
\end{proof}

\begin{lemma}\label{fixfor}
 Let $x\in X$. If $\pi$ is a $*$-homomorphism from $A$ to a unital $C^*$-algebra $A'$ such that $\pi((ev_{x}\otimes id)\alpha(f))=f(x)1$ for all $f\in B$, then
\begin{equation}\label{eq:ideal}
(\pi\otimes \pi)\Delta((ev_{x}\otimes id)\alpha(f))=f(x)1\otimes 1
\end{equation}
for all $f\in B$.
\end{lemma}
\begin{proof}
By Proposition~\ref{mx is a quotient of quotient space}, we have
$$(\pi\otimes id)\Delta((ev_{x}\otimes id)\alpha(f))=1\otimes (ev_{x}\otimes id)\alpha(f).$$
Apply $id\otimes\pi$ on both sides of the above equation, we get
\begin{align*}
(\pi\otimes \pi)\Delta((ev_{x}\otimes id)\alpha(f))
=1\otimes (ev_{x}\otimes \pi)\alpha(f)
=f(x)1\otimes 1.
\end{align*}
\end{proof}

Suppose $\mathcal{G}$ has the bounded counit $\var$. Let $I_Y$ be the ideal of $A$ generated by elements of the form
$(ev_{x}\otimes id)\alpha(f)-f(x)1$ for all $x\in Y$ and $\pi_Y: A\rightarrow A/I_Y$ be the quotient map. Note that $(ev_{x}\otimes id)\alpha(f)-f(x)1$ is in $\ker{\var}$. Hence $I_Y\subseteq \ker{\var}$ is a proper ideal of $A$.

Under the assumption of boundedness of counit,  we show the existence of stabilizer subgroups.
\begin{lemma}
$I_0$ is a Woronowicz C*-ideal.
\end{lemma}
\begin{proof}
$I_Y$ is a Woronowicz C*-ideal if and only if $I_Y\subseteq \ker{(\pi_Y\otimes \pi_Y)\Delta}$. So it suffices to show
$$(\pi_Y\otimes \pi_Y)\Delta((ev_{x}\otimes id)\alpha(f)-f(x)1)=0,$$ for all $f\in B$, which
 follows directly from Lemma~\ref{fixfor}.
\end{proof}

Hence $A/I_Y$ is a subgroup of $\mathcal{G}$, denote it by $\mathcal{H}_0$. Use $\Delta_0$ to denote the coproduct of $\mathcal{H}_0$. Then
\begin{lemma}\label{existence of subgroup fixing x}
$\mathcal{H}_0$ fixes $x$.
\end{lemma}
\begin{proof}
From the definition of $\mathcal{H}_0$, $\pi_Y((ev_{x}\otimes id)\alpha(f))=f(x)1$ for all $f\in B$. So $\mathcal{H}_0$ fixes $x$.
\end{proof}

\begin{theorem}\label{stabilizer subgroup}
$\mathcal{H}_0$ is the stabilizer of $x$.
\end{theorem}
\begin{proof}
 Suppose a subgroup $\mathcal{H}$ of $\mathcal{G}$ fixes $x$. Then $(ev_{x}\otimes id)\alpha(f)-f(x)1$ is in $\ker{\pi_{\mathcal{H}}}$. Hence $I_Y\subseteq I_{\mathcal{H}}$, which proves the existence of $\pi_{\mathcal{H},0}:A/I_Y\to A/I_{\mathcal{H}}$ and that $\pi_{\mathcal{H},0}\pi_Y=\pi_{\mathcal{H}}$.

For $a\in A$, we have
\begin{equation*}
\begin{split}
&(\pi_{\mathcal{H},0}\otimes \pi_{\mathcal{H},0})\Delta_0(a+I_Y)   \\
&=(\pi_{\mathcal{H},0}\otimes \pi_{\mathcal{H},0})(\pi_Y\otimes\pi_Y)\Delta(a)   \notag \\
&=(\pi_{\mathcal{H}}\otimes \pi_{\mathcal{H}})\Delta(a)                    \notag \\
&=\Delta_{\mathcal{H}}(a+I_{\mathcal{H}})                 \notag \\
&=\Delta_{\mathcal{H}}\pi_{\mathcal{H},0}(a+I_Y)
\end{split}
\end{equation*}
\end{proof}

\begin{lemma}
Suppose a subgroup $\mathcal{H}$ of $\mathcal{G}$ acts on $B$ by $\alpha_{\mathcal{H}}$ and fixes $x$. Under the action of $\mathcal{H}$, ${\rm Orb}_x=\{x\}$.
\end{lemma}
\begin{proof}
Suppose there exists a $y\in \rm {Orb}_x$ other than $x$. Since $\mathcal{B}$ is a dense $*$-subalgebra of $B$, there exists $0\leq f\in\mathcal{B}$ such that $f(y)=0$ and $f(x)>0$. So
$(ev_x\otimes h)\alpha_{\mathcal{H}}(f)=h(f(x)1)=f(x)=0$. Moreover it follows from the faithfulness of $h$ in $\mathcal{B}$ and $(ev_y\otimes id)\alpha_{\mathcal{H}}(f))\in\mathcal{B}$ that $(ev_y\otimes h)\alpha_{\mathcal{H}}(f)=h((ev_y\otimes id)\alpha_{\mathcal{H}}(f))>0$. This leads to a contradiction to that $(ev_y\otimes h)\alpha_{\mathcal{H}}=(ev_x\otimes h)\alpha_{\mathcal{H}}$.
\end{proof}

\section{A concrete example}
Next, we give a concrete example of stabilizer subgroups.
\begin{theorem}\label{stabilizer under quantum permutation}
Let the quantum permutation group $A_s(n)$ act on $X_n=\{x_1,x_2,...,x_n\}$ by $\alpha$ such that
$$\alpha(e_i)=\sum_{j=1}^n e_j\otimes a_{ji}$$ for $1\leq i\leq n$. For any $x\in X_n$, the stabilizer of $x$, denoted by $A_x$, is isomorphic to the quantum permutation group $A_s(n-1)$.
\end{theorem}
Before proceeding to prove the result, we first recall some facts from~\cite[Proposition 2.5]{Wang1995} about morphisms between compact quantum groups.
\begin{proposition}\label{antipodle is preserved}
Let $\Psi: A\to B$ be a morphism of compact quantum groups. That is, $(\Psi\otimes\Psi)\Delta_A=\Delta_B\Psi$. Then we have that $\Psi$ preserves the Hopf $*$-algebra structures. Namely,

$\Psi(\mathscr{A})\subseteq \mathscr{B}$, \, $\kappa_B\Psi=\Psi\kappa_A$, \, and $\varepsilon_B\Psi=\varepsilon_A$,
where for instance, $\kappa_A$ is the antipode and $\varepsilon_A$ is the counit on $\mathscr{A}$.
\end{proposition}
For convenience, let $x=x_n$ and $I_n$ be the ideal of $A_s(n)$ generated by $\{(ev_n\otimes id)\alpha(f)-f(x_n)1|f\in B\}$. If we choose $f$ to be $e_i$ for $1\leq i\leq n$, then we have $I_n$ is the ideal generated by $a_{ni}$'s and $a_{nn}-1$ for all $1\leq i\leq n-1$. Furthermore, we have
\begin{lemma}\label{stab lemma}
For all $1\leq i\leq n-1$, $a_{in}$'s are in $I_n$.
\end{lemma}
\begin{proof}
Let $\pi$ be the quotient map from $A_s(n)$ onto $A_s(n)/I_n$, which is a morphism between compact quantum groups. Let $\kappa$ and $\kappa_n$ be the corresponding antipodes of $A_s(n)$ and $A_s(n)/I_n$. By Proposition~\ref{antipodle is preserved}, we have that for  any $1\leq i\leq n-1$, $$\kappa_n\pi(a_{in})=\pi\kappa(a_{in})=\pi(a_{ni})=0.$$ Note that $a_{in}$ is in the Hopf $*$-subalgebra of $A_s(n)$. Hence Proposition~\ref{antipodle is preserved} tells us that $\pi(a_{in})$ is in the Hopf $*$-subalgebra of $A_s(n)/I_n$ on which $\kappa_n$ is injective. So  $\pi(a_{in})=0$, which proves that
$a_{in}\in I_n$ for all $1\leq i\leq n-1$.
\end{proof}
Now we are ready to prove Theorem~\ref{stabilizer under quantum permutation}.
\begin{proof}
Let $\{b_{ij}\}_{1\leq i,j\leq n-1}$ be the set of generators of $A_s(n-1)$. By the universality of $A_s(n)$, there exists a $*$-homomorphism $S: A_s(n)\to A_s(n-1)$ such that $S(a_{ij})=b_{ij}$ for $1\leq i,j\leq n-1$, otherwise $S(a_{ij})=\delta_{ij}$. Since $a_{ni}$'s and $a_{nn}-1$ are in $\ker{S}$ for all $1\leq i\leq n-1$, it follows  that $I_n\subseteq \ker{S}$. Therefore, naturally $S$ induces a $*$-homomorphism $\Phi:A_s(n)/I_n\to A_s(n-1)$.

On the other hand, by the universality of $A_s(n-1)$ and Lemma~\ref{stab lemma}, there exists a $*$-homomorphism $\Psi: A_s(n-1)\to A_s(n)/I_n$ such that $\Psi(b_{ij})=a_{ij}+I_n$. It is easy to check that $\Psi=\Phi^{-1}$, which proves that $\Phi$ is an isomorphism and ends the proof.
\end{proof}
\begin{remark}
\begin{enumerate}
\item In fact, the quantum subgroup stabilizing a character $\chi$ is already introduced by C. Pinzari~\cite[Theorem 7.3]{Pinzari2007} without assuming the boundedness of the counit. However, this assumption could not be omitted since we need this to guarantee that the ideal generated by $(\chi\otimes id)\alpha(b)-\chi(b)1_A$ is not $A$.
\item Theorem~\ref{stabilizer under quantum permutation} is first stated (without proof) by S. Wang in~\cite[Concluding remarks (2) preceding Appendix]{Wang1998}.
\end{enumerate}
\end{remark}

\end{document}